\documentclass[letterpaper]{article}

\usepackage{setspace}
\usepackage[english]{babel}
\usepackage[utf8x]{inputenc}
\usepackage[T1]{fontenc}
\usepackage[top=1.3in, left=1.3in, right=1.3in, bottom=1.1in]{geometry}
\usepackage[numbers]{natbib}

\usepackage{amsmath}
\usepackage{amssymb}
\usepackage{graphicx}
\usepackage{amsthm}
\usepackage{url}
\RequirePackage{color}
\usepackage[pdfstartview=FitH,
colorlinks,
linkcolor=reference,
citecolor=citation,
urlcolor=e-mail]{hyperref}
\definecolor{e-mail}{rgb}{0,.40,.80}
\definecolor{reference}{rgb}{.20,.60,.22}
\definecolor{citation}{rgb}{0,.40,.80}

\newlength{\bibitemsep}
\newlength{\bibparskip}
\let\oldthebibliography\thebibliography
\renewcommand\thebibliography[1]{
  \oldthebibliography{#1}
  \setlength{\parskip}{\bibitemsep}
  \setlength{\itemsep}{\bibparskip}
}

\DeclareMathOperator{\prem}{prem}
\DeclareMathOperator{\class}{class}
\DeclareMathOperator{\lc}{lc}
\DeclareMathOperator{\rep}{rep}
\DeclareMathOperator{\zero}{zero}
\DeclareMathOperator{\stab}{stab}
\DeclareMathOperator{\Gal}{Gal}

\newtheorem{thm}{Theorem}[section]
\newtheorem{lem}{Lemma}[section]

\newtheorem{cor}{Corollary}[section]
\newtheorem{example}{Example}[section]
\newtheorem*{rem}{Remark}
\newtheorem*{ack}{Acknowledgement}
\newtheorem{definition}{Definition}[section]

\providecommand{\keywords}[1]{{\textit{keywords:}} #1}

\title{A New Bound on Hrushovski's Algorithm for Computing the Galois Group of a Linear Differential Equation}
\author{Mengxiao Sun}
\date{}

\begin{document}
\maketitle

\begin{abstract}
The complexity of computing the Galois group of a linear differential equation is of general interest. In a recent work, Feng gave the first degree bound on Hrushovski's algorithm for computing the Galois group of a linear differential equation. This bound is the degree bound of the polynomials used in the first step of the algorithm for finding a proto-Galois group (see Definition \ref{def: protoGalois}) and is sextuply exponential in the order of the differential equation. In this paper, we use Sz\'ant\'o's algorithm of triangular representation for algebraic sets to analyze the complexity of computing the Galois group of a linear differential equation and we give a new bound which is triple exponential in the order of the given differential equation.
\\
\end{abstract}

\section{Introduction}

The differential Galois group (see Definition \ref{def: GaloisGroup}) is an analogue for a linear differential equation of the classical Galois group for a polynomial equation. 
An important application of the differential Galois group is that a linear differential equation can be solved by integrals, exponentials and algebraic functions if and only if the
connected component of its differential Galois group is solvable \cite{kolchin1948algebraic,van2003galois}.
For example \cite[Appendix]{kolchin1968algebraic}, the differential Galois group of Bessel's equation $t^2y''+ty'+(t^2- \nu^2)y=0$ over $\mathbb{C}(t)$ is isomorphic to $SL_2(\mathbb{C})$ (not solvable) when $\nu \not\in \frac{1}{2}+\mathbb{Z}$. In other words, Bessel's equation cannot be solved by integrals, exponentials and algebraic functions unless $\nu \in \frac{1}{2}+\mathbb{Z}$.
Computing the differential Galois groups would help us determine the existence of the solutions expressed in terms of elementary functions (integrals, exponentials and algebraic functions) and understand the algebraic relations among the solutions. 

Hrushovski in \cite{hrushovski2002computing}  first proposed an algorithm for computing the differential Galois group of a general linear differential equation over $k(t)$ where $k$ is a computable algebraically closed field of characteristic zero.
Recently, Feng approached finding a complexity bound of the algorithm in \cite{feng2015hrushovski}, which is the degree bound of the polynomials used in the first step of the algorithm for finding a proto-Galois group (see Definition \ref{def: protoGalois}), but not for the whole algorithm. The bound given by Feng is sextuply exponential in the order $n$ of the differential equation.

In this paper, we present a triple exponential degree bound using triangular sets instead of Gr\"obner bases for representing the algebraic sets. 
In general, the degrees of defining equations of a differential Galois group cannot be bounded by a function of $n$ only. For example \cite[Example~1.3.7, page~12]{singer2007introduction}, the differential Galois group of $y'= \frac{1}{mt}y$ over $\mathbb{C}(t)$ is isomorphic to $\mathbb{Z}/m\mathbb{Z}$ where $m$ is a positive integer, which implies that the degree of the defining equation $x^m-1$ is $m$.

A crucial point of Hrushovski's algorithm is that one can find a proto-Galois group which is an algebraic subgroup of $GL_n(k)$, provided that the degree bound of the defining equations of the proto-Galois group is computed.
The differential Galois group can then be recovered from the proto-Galois group (more details in \cite{feng2015hrushovski,hrushovski2002computing}).
Therefore, a bound for the proto-Galois group plays an important role in determining the complexity of Hrushovski's algorithm. 
Following Feng's approach, we prove that such a proto-Galois group exists by constructing a family $\mathcal{F}$ of algebraic subgroups such that the identity component of any algebraic subgroup $H' \subseteq GL_n(k)$ is contained in some $H$ of $\mathcal{F}$ and $[H'H:H]$ is uniformly bounded. 
We also prove that the degrees of the defining equations of any element of $\mathcal{F}$ are bounded by $\bar{d}$ depending on the order $n$ of the given differential equation. This is stated as Theorem \ref{thm:degreekappa3}.
Then by collecting the algebraic subgroups $\bar{H}$ such that there is some $H$ of $\mathcal{F}$ such that $[\bar{H}:H] \le \bar{d}$, we obtain a family $\bar{\mathcal{F}}$ of algebraic subgroups in which one can always find a proto-Galois group for any linear differential equation. Moreover, we give a numerical degree bound of the defining equations of any algebraic subgroup of $\mathcal{\bar{F}}$. This is stated as Corollary \ref{cor: degreebound}.

Using degrees of defining equations of algebraic subgroups to bound $\bar{\mathcal{F}}$, one needs an upper degree bound of the defining equations of the algebraic subgroups of $\mathcal{F}$ and an upper bound of $[\bar{H}:H]$. 
Hence, a double-exponential degree bound for computing Gr\"obner bases would be involved if one chooses to represent an algebraic subgroup by the generating set of its defining ideal (generated by the defining equations).
In order to give a better bound, we represent an algebraic subgroup by the triangular sets (see Definition \ref{def: trirepofalgebraicsubgps}) instead of the generating set in the process of constructing $\mathcal{F}$.
In such a process, we need to take the differences between Gr\"obner bases and triangular sets into account.
We apply Sz\'ant\'o's modified Wu-Ritt type decomposition algorithm \cite{szanto1997complexity,SzantoThesis} which has been proved to be more efficient than computing a Gr\"obner basis and make use of the numerical bound for Sz\'ant\'o's algorithm \cite{amzallag2016complexity} to adapt to the complexity analysis of Hrushovski's algorithm. 
In doing this, we are able to avoid working with Gr\"obner bases to get a better bound of the degrees of the defining equations of the algebraic subgroups of $\mathcal{F}$ which is triple exponential in the order $n$ of the given differential equation. 
Additionally, we are able to not increase the degree bound of the defining equations of the algebraic subgroups of $\bar{\mathcal{F}}$.
Each element $\bar{H}$ of $\bar{\mathcal{F}}$ is a union of at most $\bar{d}$ cosets of some element $H$ of $\mathcal{F}$.
The degree bound for the generating set of the ideal generated by the defining equations of $\bar{H}$ would be raised to an exponent at most $\bar{d}$, which results in a big increase on the degrees of the defining equations of the algebraic subgroups of $\bar{\mathcal{F}}$.
However, this issue has been resolved when expressing the algebraic subgroups by triangular sets. 
More details can be found in sections \ref{sec: 2} and \ref{sec: 3}.

Besides Hrushovski's general algorithm, there are other algorithmic results in the Galois theory of linear differential equations. Kovacic in \cite{kovacic1986algorithm} presented an algorithm for computing the Galois group of a second order linear differential equation. 
The Galois groups of second and third order linear differential equations were studied by Singer and Ulmer in \cite{singer1993galois}. Compoint and Singer in  \cite{compoint1999computing} proposed an algorithm for computing the Galois group if the differential equation is completely reducible. 
The numeric-symbolic computation of differential Galois groups was presented by van der Hoeven in \cite{van2007around}.
For more details on the differential Galois theory from an algorithmic point of view,  readers are referred to \cite{van2005galois} and \cite{singer2007introduction}.

This paper is organized as follows. In sections \ref{sec: 2.1} and \ref{sec: 2.2}, we introduce the notations, definitions and facts from triangular sets and differential Galois groups. 
In section \ref{sec: 3.1}, we state and prove the preparation lemmas which we use in analyzing the complexity of the algorithm.
In section \ref{sec: 3.2}, we present and prove the new complexity bound of Hrushovski's algorithm. In section \ref{sec: 4}, we compare our bounds when $n=2$ with the ones in \cite[Proposition~B.11, Proposition~B.14]{feng2015hrushovski}.

\section{Preliminaries}
\label{sec: 2}

\subsection{Triangular sets and Sz\'ant\'o's algorithm}
\label{sec: 2.1}

\begin{definition}
Fix a monomial ordering with $x_1<x_2< \dots <x_n$ in a polynomial ring $k[x_1, \dots, x_n]$. Let $f \in k[x_1, \dots, x_n]$. Then $\class{(f)}$ denotes the highest indeterminate in $f$.

\end{definition}

\begin{definition}

Fix a monomial ordering with $x_1<x_2< \dots <x_n$ in a polynomial ring $k[x_1, \dots, x_n]$. A sequence of polynomials $\{g_1, \dots, g_m\}$ is called a triangular set if $\class{(g_i)}<\class{(g_j)}$ for all $i<j$.

\end{definition}

The pseudo division is a method of division for multivariate polynomials which is a generalization of the method of division for univariate polynomials.

Let $f \in k[x_1, \dots, x_n]$ and $G = \{g_1, \dots, g_m\}$ be a triangular set in $k[x_1, \dots, x_n]$. We denote the leading coefficient of $g_i$ by $\lc{(g_i)}$. There exist polynomials $q_1, \dots, q_m$  and $\alpha_1, \dots, \alpha_m \in \mathbf{N}$ such that $\lc{(g_1)}^{\alpha_1} \dots \lc{(g_m)}^{\alpha_m}f = \Sigma_{i=1}^m q_i g_i + f_0$ where $\deg_{x_{j_i}} (f_0)< \deg_{x_{j_i}} (g_i)$ if $\class{(g_i)} = x_{j_i}$ for $1 \le i \le m$.
We call $f_0$ the pseudo remainder of $f$ by $G$, denoted by $\prem{(f,G)}$. If $f=\prem{(f,G)}$, then we say that $f$ is reduced modulo $G$.

\begin{definition}
A triangular set $G \subseteq k[x_1, \dots, x_n]$ represents an ideal $I$ if $I=\{f \in k[x_1, \dots, x_n] : \prem{(f, G)} = 0 \}$. We denote the ideal represented by $G$ by $\rep{(G)}$.
\end{definition}

In general, a triangular set is not a generating set of the ideal which it represents and has more zeros than the ideal. Consider a triangular set $G=\{xy\} \subseteq \mathbb{C}[x,y]$ with $x<y$. $G$ has a zero $(0,-2)$ which is not in $\zero{(\rep(G))}$ because $y \in \rep(G)$.

Note that $\rep(G)$ is not necessarily an ideal. Consider a triangular set $G=\{x, xy\} \subseteq \mathbb{C}[x,y]$ with $x<y$. $y$ and $y+1$ are in $\rep(G)$, but $y-(y+1)= -1$ is not in $\rep(G)$.

Sz\'ant\'o in \cite{szanto1997complexity} and \cite{SzantoThesis} presented a Wu-Ritt type decomposition
algorithm expressing the radical of an
ideal as an intersection of radical
ideals represented by unmixed triangular sets. 
An unmixed triangular set is a triangular set with some certain conditions. For our purposes, we state the following theorem without emphasizing the "unmixed" property of triangular sets.

\begin{thm} 
\cite[Theorem~4.1.7]{SzantoThesis}
\label{thm:Szanto} 
Let $I \subseteq k[x_1, \dots, x_n]$ be an ideal. Then there is an algorithm which computes a triangular representation of $I$, that is computing a family of triangular sets $G_i$ such that $\sqrt{I}=\bigcap^m_{i=1} I_i$ where $I_i = \rep{(G_i)}$ for each $i$.
\end{thm}

\begin{definition}
An ideal $I \subseteq k[x_1, \dots, x_n]$ is bounded by $d$ if any polynomial $p_j$ in the generating set $\{p_1, \dots, p_m\}$ of $I$ has degree not greater than $d$.
\end{definition}

In order to get a numerical complexity bound for Hrushovski's algorithm, we need a numerical complexity bound for Sz\'ant\'o's algorithm, which is stated as the following theorem.

\begin{thm} \cite[Theorem~3.5]{amzallag2016complexity}
\label{thm:SzantoBound} 
Suppose that $n>1$ and $I \subseteq k[x_1,  \dots, x_n]$ is an ideal bounded by $d$. Then all the polynomials appearing during the computations and in the output of Sz\'ant\'o's algorithm have degrees not greater than $nd^{5.5n^3}$.
\end{thm}

\subsection{Differential Galois groups and Hrushovski's algorithm}
\label{sec: 2.2}

We consider a linear differential equation in the matrix form:

\begin{displaymath}
\delta{(Y)}=AY \tag{1} \label{equation:1}
\end{displaymath}
where $Y$ is a vector containing $n$ unknowns and $A$ is an $n×n$ matrix with entries in $k(t)$. 
Denote the Picard-Vessiot extention field of the differential field $k(t)$ by $K$ with the derivation $\delta=\frac{d}{dt}$ and the solution space of $(\ref{equation:1})$ by $V$ in $K$. Let $F \in GL_n(K)$ be a fundamental matrix of (\ref{equation:1}). Let $GL(V)$ be the group of automorphisms of the solution space $V$. Then there is a group isomorphism $\Phi_F$: $GL(V) \rightarrow GL_n(k)$ sending $\sigma \in GL(V)$ to $M_{\sigma} \in GL_n(k)$ where $FM_{\sigma}=\sigma(F)$.

\begin{definition} 
\label{def: GaloisGroup}
The Galois group $\mathcal{G}$ of \textnormal{(\ref{equation:1})} is the group of $k(t)$-automorphisms of $K$ which commutes with the  derivation and fixes $k(t)$ pointwise.
\end{definition}

\begin{definition}
An algebraic subgroup $H$ of $GL_n(k)$ is bounded by $d$ if there exist finitely many polynomials $p_1, \dots, p_m \in k[x_{i,j}]_{1<i,j<n}$ of degrees not greater than $d$ such that $H= \zero{(p_1, \dots, p_m)} \cap GL_n(k)$. 
\end{definition}

Let $H \subseteq GL_n(k)$ be an algebraic subgroup.
Let $H^{0}$ be the identity component of $H$ and $\Phi_F{(\mathcal{G})}^{0}$ be identity component of $\Phi_F{(\mathcal{G})}$. Let $(H^{0})^t$ be the intersection of kernels of all characters of $H^{0}$.

The definition of a proto-Galois group of (\ref{equation:1}) was introduced by Feng in \cite{feng2015hrushovski}, which is as follows:

\begin{definition}[{\cite[Definition 1.1]{feng2015hrushovski}}] \label{def: protoGalois}
If there is an algebraic subgroup $H$ of $GL_n(k)$ such that
\begin{displaymath}
(H^{0})^t\trianglelefteq \Phi_F(\mathcal{G})^{0} \subseteq \Phi_F(\mathcal{G}) \subseteq H,
\end{displaymath}
then $H$ is called a \textnormal{proto-Galois group} of $(\ref{equation:1})$.
\end{definition}
In Hrushovski's algorithm, one can compute an integer $\tilde{d}$ such that there is a proto-Galois group $H$ of $GL_n(k)$ bounded by $\tilde{d}$. The bound $\tilde{d}$ is given by Feng in \cite{feng2015hrushovski}.

\begin{example}
\label{Example1}
Consider the first order linear differential equation $y'= \frac{1}{3t}y$ over $\mathbb{C}(t)$. 
The differential Galois group is the subgroup $\{1, -\frac{1}{2}+\frac{\sqrt3}{2}i, -\frac{1}{2}-\frac{\sqrt3}{2}i \} \subseteq \mathbb{C}^{*}$, where $\mathbb{C}^{*}$ is the multiplicative group of complex numbers \cite[Example~1.3.7, page~12]{singer2007introduction}. 
The identity component of $\mathbb{C}^{*}$ is itself.
The intersection of kernels of all characters of $\mathbb{C}^{*}$ is trivial because the identity map of $\mathbb{C}^{*}$ is a character. 
So in this case $\mathbb{C}^{*}$ is a proto-Galois group. 
\end{example}

\begin{example}
\label{Example2}
Consider the Airy equation $y''=ty$ over $\mathbb{C}(t)$. 
The differential Galois group is the subgroup $SL_2(\mathbb{C}) \subseteq GL_2(\mathbb{C})$ \cite[Example~8.15, page~250]{van2003galois}.
The identity component of $SL_2(\mathbb{C})$ is itself because $SL_2(\mathbb{C})$ is connected.
The identity component of $GL_2(\mathbb{C})$ is itself because $GL_2(\mathbb{C})$ is connected.
A character $\phi$ of $GL_2(\mathbb{C})$ is of the form $\forall g \in GL_2(\mathbb{C})$ $\phi(g)=(\det{(g)})^{n}$ where $\det$ is the determinant map and $n \in \mathbb{N}$.
So the intersection of kernels of all characters of $GL_2(\mathbb{C})$ is $SL_2(\mathbb{C})$.
So $GL_2(\mathbb{C})$ is a proto-Galois group.
From Definition \ref{def: protoGalois}, it is not hard to see the differential Galois group itself is a proto-Galois group.
\end{example}

\begin{example}
\label{Example3}
Consider the second order linear differential equation $y''+\frac{1}{t}y'=0$ over $\mathbb{C}(t)$. 
The differential Galois group is 
\[ 
\bigg \{ \begin{pmatrix} 
1 & c \\
0 & 1 
\end{pmatrix} : c \in \mathbb{C} \bigg \} 
\] 
\cite[Example~4.1, page~90]{crespo2007introduction}.
The same analysis in Example \ref{Example2} shows that the intersection of kernels of all characters of $GL_2(\mathbb{C})$ is $SL_2(\mathbb{C})$. But $GL_2(\mathbb{C})$ is not contained in the identity component of the differential Galois group.
So in this case $GL_2(\mathbb{C})$ cannot be a proto-Galois group.
The intersection of kernels of all characters of $SL_2(\mathbb{C})$ is itself which is not contained in the identity component of the differential Galois group.
So in this case $SL_2(\mathbb{C})$ cannot be a proto-Galois group.
\end{example}

\begin{example}
\label{Example4}
Consider the second order linear differential equation $y''-2ty'-2y=0$ over $\mathbb{C}(t)$. 
The differential Galois group is
\[ 
\bigg \{ \begin{pmatrix} 
a & c \\
0 & b 
\end{pmatrix} : a, b, c \in \mathbb{C}, ab \ne 0 \bigg \} 
\]
\cite[Example~6.10, pages~81,82,83]{Magid94lectureson}.
The same analysis in Example \ref{Example3} shows that
in this case $GL_2(\mathbb{C})$ cannot be a proto-Galois group.
The differential Galois group in this case is not even a subgroup of $SL_2(\mathbb{C})$,
so $SL_2(\mathbb{C})$ cannot be a proto-Galois group.
\end{example}

\begin{rem}
The proto-Galois group of a linear differential equation is not unique. As shown in Examples \ref{Example1} and \ref{Example2},  
the proto-Galois group can be far from the differential Galois group. But a group being large does not make it a proto-Galois group as shown in Examples \ref{Example3} and \ref{Example4}.
\end{rem}

Hrushovski in \cite[Corollary~3.7]{hrushovski2002computing} proved that one can compute an integer $d_3$ such that  there is a proto-Galois group $H$ of $GL_n(k)$ bounded by $d_3$. Feng in \cite[Propostion B.14]{feng2015hrushovski} gave the first explicit bound for $d_3$ which is sextuply exponential in the order $n$ of the given linear differential equation.

To understand the key role of the integer $d_3$ in analyzing the complexity of Hrushovski's algorithm, we separate the algorithm in three main steps following the way in which Feng in \cite{feng2015hrushovski} described it.

\begin{definition}
Let $\tilde{V}=\{Fh : h\in GL_n(k)\}$. A subset $V_0$ of $\tilde{V}$ is defined by finitely many polynomials $p_1, \dots, p_m$ if $V_0 = \zero{(p_1, \dots, p_m)} \cap \tilde{V}$ where $\zero{(p_1, \dots, p_m)}$ denotes the zero set of $\{p_1, \dots, p_m\}$ in $k^{n×n}$. If $p_1, \dots, p_m$ have coefficients in $k$, we say that $V_0$ is $k$-definable subset of $\tilde{V}$.
\end{definition}

\begin{itemize}
\item In the first step, we compute a proto-Galois group $H$ of (\ref{equation:1}) bounded by $d_3$. The existence of $H$ is guaranteed by {\cite[Corollary 3.7]{hrushovski2002computing}}. Let $N_{d_3}(\tilde{V})$ be the set of all subsets of $\tilde{V}$ defined by finitely many polynomials of degrees not greater than $d_3$.
Then one can compute $H$ by the intersection of the stabilizers of $k$-definable elements in $N_{d_3}(\tilde{V})$.

\item In the second step, we compute the identity component $(\Phi_F(\mathcal{G}))^0$ of $\Phi_F(\mathcal{G})$. Let $\chi_1, \dots, \chi_l$ be the generators of the character group of $\big{(} \Phi_F(\mathcal{G})\big{)}^0$. Let $\hat{k}$ be an algebraic extension of $k(t)$. Then $\chi\big{(} \Phi_F(\mathcal{G})^0\big{)}$ is the Galois group of some exponential extension $\hat{K}$ of $\hat{k}$ where $\chi = (\chi_1, \dots, \chi_l)$. $\hat{K}$ can be obtained by computing hyperexponential solutions of some symmetric power system of (\ref{equation:1}). $\big{(}\Phi_F(\mathcal{G})\big{)}^0$ can be found by the pre-image of $\chi\big{(} \Phi_F(\mathcal{G})^0\big{)}$ in $\big{(}\Phi_F(\mathcal{H})\big{)}^0$.

\item In the last step, we compute the differential Galois group $\mathcal{G}$ of (\ref{equation:1}).
Let $\mathcal{G}^{0}$ be the pre-image of $\big{(} \Phi_F(\mathcal{G})\big{)}^0$.
Find a Galois extension $k_G$ of $k(t)$ and a $k_G$-definable subset $V_{k_G}$ of $\tilde{V}$ such that $\mathcal{G}^0$ = $\stab{(V_{k_G})}$ where $\stab{(V_{k_G})}$ is the stabilizer of $V_{k_G}$. Then 
\begin{displaymath}
\mathcal{G} = \bigcup_{i=0}^{m}{\{ \sigma \in GL(V)|\sigma(V_{k_G})=V_i\}}
\end{displaymath}
where $V_i$ is the orbit of $V_{k_G}$ under the action of $\Gal{\big{(}k_G/k(t)\big{)}}$.

\end{itemize}

From the first step of the algorithm, we can see that the integer $d_3$ determines the complexity of computing a proto-Galois group. The differential Galois group is obtained by recovering the proto-Galois group in the next two steps. Therefore $d_3$ plays an important role in determining the complexity of the whole algorithm.

\section{Complexity bound on Hrushovski's algorithm}
\label{sec: 3}

\subsection{Preparation}
\label{sec: 3.1}

\begin{definition}
We say that an ideal $I \subseteq k[x_1, \dots, x_n]$ has a triangular representation if $\sqrt{I}$ is expressed by an intersection of radical ideals $I_i$ such that for each $i$, $I_i=\rep{(G_i)}$ where $G_i$ is a triangular set in $I_i$. A triangular representation of $I$ is bounded by $d$ if every polynomial in $G_i$ has degree not greater than $d$.
\end{definition}

\begin{definition}
\label{def: trirepofalgebraicsubgps}
An algebraic subgroup $H \subseteq GL_{n}(k)$ is said to have a triangular representation if the ideal generated by the defining equations of $H$ has a triangular representation. A triangular representation of $H$ is bounded by $d$ if every polynomial in the triangular representation has degree not greater than $d$.
\end{definition}

Let $I \subseteq k[x_1, \dots, x_n]$ be an ideal.
In \cite[Proposition~B.2]{feng2015hrushovski}, Feng gave a degree bound for $I \cap k[x_1, \dots, x_r]$ which is double-exponential in $n$ using the computation of Gr\"obner bases. In the following lemma, we give a degree bound for the triangular representation of $I \cap k[x_1, \dots, x_r]$ which is polynomial exponential in $n$.

\begin{lem}
\label{lem:elimination}
Assume that $n>1$. Let $I \subseteq k[x_1, \dots, x_n]$ be an ideal bounded by $d$ and $1 \le r \le n$. Then $I \cap k[x_1, \dots, x_r]$ has a triangular representation bounded by $nd^{5.5n^3}$. 

\end{lem}

\begin{proof}
Assume that $\sqrt{I} = \bigcap_i{I_i}$ is a triangular representation of $I$ where $I_i = \rep{(G_i)}$ and $G_i$ is a triangular set of $I_i$ .
Since 
\begin{displaymath}
\sqrt{I} \cap k[x_1, \dots, x_r]= \sqrt{I \cap k[x_1, \dots, x_r]}, 
\end{displaymath}
it suffices to show that for each $i$
\begin{displaymath}
\rep{\big{(} G_i \cap k[x_1, \dots, x_n]\big{)}} \cap k[x_1, \dots, x_r]= \rep{(G_i)} \cap k[x_1, \dots, x_r] \tag{2} \label{equation:2}
\end{displaymath}
where $\rep{\big{(} G_i \cap k[x_1, \dots, x_n]\big{)}}$ is an ideal in $k[x_1, \dots, x_n]$.
Let $g \in$ LHS of (\ref{equation:2}).
Then 
\begin{displaymath}
\prem{\big{(} g, G_i \cap k[x_1, \dots, x_r]\big{)}}=0. 
\end{displaymath}
If $G_i \subseteq k[x_1, \dots, x_r]$, then 
\begin{displaymath}
\prem{(g, G_i)} = \prem{\big{(} g, G_i \cap k[x_1, \dots, x_r]\big{)}}. 
\end{displaymath}
So $\prem{(g, G_i)} = 0$.
If $G_i \not\subseteq k[x_1, \dots, x_r]$, then $G_i$ must have at least one polynomial containing terms larger than $x_r$. 
Let $G_i = \{g_{i,1}, \dots, g_{i,s}\}$ and assume that $g_{i,j+1}, \dots, g_{i,s}$ contain terms larger than $x_r$. 
Then
\begin{displaymath}
G_i \cap k[x_1, \dots, x_r]=\{g_{i,1}, \dots, g_{i,j}\}
\end{displaymath}
with $j<s$.
Since $g \in k[x_1, \dots, x_r]$, $g$ is reduced modulo $g_{i,j}, \dots, g_{i,s}$. 
Then $\prem{(g, G_i)}=0$. 
So $g \in$ RHS of (\ref{equation:2}). 
Let $f \in$ RHS of (\ref{equation:2}). 
Then $f \in k[x_1, \dots, x_r]$ and $\prem{(f, G_i)} = 0$. Since $f \in k[x_1, \dots, x_r]$, $f$ is reduced modulo polynomials containing terms larger than $x_{r}$ in $G_i$. 
So 
\begin{displaymath}
\prem{\big{(} f, G_i \cap k[x_1, \dots, x_r]\big{)}} = 0.
\end{displaymath}
So $f \in$ LHS of (\ref{equation:2}).
Therefore, $I \cap k[x_1, \dots, x_r]$ has a triangular representation which is 
\begin{displaymath}
\sqrt{I\cap k[x_1, \dots, x_r]} = \bigcap_i{I_i'}
\end{displaymath}
where 
\begin{displaymath}
I_i' = \rep{\big{(} G_i \cap k[x_1, \dots, x_r]\big{)}}
\end{displaymath}
where $\rep{\big{(} G_i \cap k[x_1, \dots, x_r]\big{)}}$ is an ideal in $k[x_1, \dots, x_r]$.
By Theorem \ref{thm:Szanto}, the triangular representation of $I \cap k[x_1, \dots, x_r]$ is bounded by $nd^{5.5n^3}$.
\end{proof}

\begin{lem} \label{lem: IJbound}
Let $I, J \subseteq k[x_1, \dots, x_n]$ be ideals. Assume that $I$ and $J$ have triangular representations bounded by $d$. Then $IJ$ has a triangular representation bounded by $d$. 
\end{lem}

\begin{proof}
Suppose that $\sqrt{I} = \bigcap_i{I_i}$ and $\sqrt{J} = \bigcap_j{J_j}$ are triangular representations of $I$ and $J$.
Then $\sqrt{IJ} = \sqrt{I} \cap \sqrt{J} = {\big{(} \bigcap I_i\big{)}} \cap {\big{(} \bigcap J_j\big{)}}$.
So if $I_i = \rep{(G^{I_i})}$ and $J_j = \rep{(G^{J_j})}$ for some triangular sets $G^{I_i}$ and $G^{J_j}$, 
then $\sqrt{IJ} = \big{(} \bigcap \rep{\big{(} G^{I_i})}\big{)} \cap \big{(} \bigcap \rep{(G^{J_j})}\big{)}$.
Therefore, $IJ$ has a triangular representation bounded by $d$.
\end{proof}

\begin{definition}
We say that a family $\mathcal{F}$ of algebraic subgroups in $GL_n(k)$ is represented by a family of triangular sets if any $H \in \mathcal{F}$ has a triangular representation, and $\mathcal{F}$ is bounded by $d$ if the triangular representations of any $H \in \mathcal{F}$ are bounded by $d$.
\end{definition}

Let $H \subseteq GL_n(k)$ be an algebraic subgroup. Let $\tau : H \longrightarrow GL_l (k)$ be a homomorphism where $l$ is a positive integer. Assume that $\tau = (\frac{P_{i,j}}{Q})$, where $P$ and $Q$ are polynomials with coefficients in $k$ and $1 \le i,j \le l$. The homomorphism $\tau$ is said to be bounded by $d$ if the polynomials $P_{i,j}$ and $Q$ have degrees not greater than $d$.

In \cite[Lemma~B.5]{feng2015hrushovski}, Feng gave a degree bound for the generating set of the ideal generated by the defining equations of $\tau^{-1}\big{(} H' \cap \tau(H)\big{)}$ where $H \subseteq GL_n(k)$ and $H' \subseteq GL_l(k)$. We use a similar argument to give in the following lemma a bound for the triangular representation of $\tau^{-1}\big{(} H' \cap \tau(H)\big{)}$.

\begin{lem} \label{lem: inversebound}
Assume that $n>1$. Let $H \subseteq GL_n(k)$ be an algebraic  subgroup whose triangular representation bounded by $d$ and $H' \subseteq GL_l(k)$ be an algebraic subgroup whose triangular representation bounded by $d'$. Assume that the homomorphism $\tau : H \longrightarrow GL_{l}(k)$ is bounded by $m$.
Then $\tau^{-1}\big{(} H' \cap \tau(H)\big{)}$ has a triangular representation bounded by $n\big{(}\max{(d, md')}\big{)}^{5.5n^3}$.
\end{lem}

\begin{proof}
Let $I(H)$ be the ideal generated by the defining equations of $H$ and $I(H')$ be the ideal generated by the defining equations of $H'$. 
Let $X$ be the set of indeterminates $x_{\alpha,\beta}$, $1 \le \alpha, \beta \le n$ and Y be the set of indeterminates $y_{\zeta,\eta}$, $1 \le \zeta, \eta \le l$.
Assume that $\tau=(\frac{P_{i,j}}{Q})_{1\le i,j\le l}$ where $P_{i,j}$ and $Q$ are polynomials in $k[X]$.
Assume that $H$ has a triangular representation
\begin{displaymath}
\sqrt{I(H)} =\bigcap_{r}\rep{(G_r)}
\end{displaymath}
and $H'$ has a triangular representation 
\begin{displaymath}
\sqrt{I(H')}= \bigcap_{w}\rep{(F_w)},
\end{displaymath}
where $G_{r}$ are triangular sets in $k[X]$ and $F_{w}$ are triangular sets in $k[Y]$.
$\{G_r\}$ and $\{F_w\}$ in the triangular representations of $H$ and $H'$ computed by Sz\'ant\'o's algorithm are unmixed triangular sets (see \cite[Proposition 6]{szanto1997complexity}) which guarantee 
$\zero{(\bigcup_w F_w)} = H'$ and $\zero{(\bigcup_r G_r)}= H$. 
A subroutine called \textbf{unmixed} can transform a triangular set to an unmixed one (see \cite[Section 4.2]{SzantoThesis}  for more details).
Since $\zero{(\bigcup_w F_w)} = H'$, composing every polynomial in each $F_w$ with $\tau$ and clearing the denominators, 
we can get the sets $E_w$ of polynomials in $k[X]$ such that $\zero{(\bigcup_w E_w)}=\tau^{-1}(H')$.
Since $\zero{(\bigcup_r G_r)}= H$,
$\zero{\big{(} (\bigcup_r G_r) \cup (\bigcup_w E_w)\big{)}}= \tau^{-1}\big{(} H'\cap\tau(H)\big{)}$.
Let $J$ be the ideal generated by $\big{(} (\bigcup_r G_r) \cup (\bigcup_w E_w)\big{)}$.
Thus, $\zero{(J)} =\tau^{-1}\big{(} H'\cap\tau(H)\big{)}$.
Since the degrees of polynomials in $G_r$ are not greater than $d$ and the degrees of polynomials in $E_w$ are not greater than $md'$,
by Theorem \ref{thm:SzantoBound}, $J$ has a triangular representation bounded by $n\big{(}\max{(d, md')}\big{)}^{5.5n^3}$. 
That is, $\tau^{-1}\big{(} H'\cap\tau(H)\big{)}$ has a triangular representation bounded by $n\big{(}\max{(d, md')}\big{)}^{5.5n^3}$.
\end{proof}

In \cite[Proposition~B.6]{feng2015hrushovski}, Feng uniformly bounded the homomorphisms defined in the following lemma by considering the bound for the generating set of the ideal generated by the defining equations of an algebraic subgroup. Instead, we present a bound for such homomorphisms by making use of the bound for the triangular representation of an algebraic subgroup. The proof is similar to the one in \cite[Proposition~B.6]{feng2015hrushovski}.

\begin{lem} \label{lem:boundhom}
Assume that $n>1$. Let $H$ and $H'$ be algebraic subgroups of $GL_n(k)$ such that $H \unlhd H'$. Assume that $H$ has a triangular representation bounded by $d$. Then there exists a homomorphism 
\begin{displaymath}
\tau_{H',H} : H' \longrightarrow GL_{d^{*}}(k)
\end{displaymath}
bounded by $n^{*}$ 
with $\ker(\tau_{H',H})=H$, $d^{*}= \max\limits_{i} \big\{ {{n^2+d \choose d} \choose i}^2 \big\}$ and $n^{*} = d^{*}d{{n^2+d} \choose d}$.
\end{lem}

\begin{proof}
The existence of such a homomorphism is guaranteed by \cite[Theorem, page 82]{humphreys}.
Let $G(H)$ be the family of triangular sets in a triangular representation of $H$. Then $G(H)$ is a $k$-vector space with a finite dimension. 
Let $k[x_{i,j}]_{ \le d}$ be the set of polynomials of degrees not greater than $d$ where $1 \le i,j \le n$.  
Let $I(H)=\{P(x_{i,j})\in k[x_{i,j}]_{\le d}|P(H)=0\}$ and $l=\dim_{k}(I(H))$.
Let 
\begin{displaymath}
E= \bigwedge^l k[x_{i,j}]_{ \le d}, 1 \le i,j \le n
\end{displaymath}
which is the $l$th exterior power of $k[x_{i,j}]_{\le d}$.
Since $k[x_{i,j}]_{ \le d}$ is a $k$-vector space with dimension ${ n^2+d} \choose d$, 
$\dim_{k}(E) = {{n^2+d \choose d} \choose l}$ and $\bigwedge^l C(H) = kv$ for some $v \in E$ where  $\bigwedge^l C(H)$ is the $l$th exterior power of $C(H)$. 
By a similar argument in the proof of \cite[Lemma~B.6]{feng2015hrushovski}, we can construct a desired homomorphism bounded by $n^*$.
\end{proof}

Let $U$ be a subgroup generated by unipotent elements of $GL_n(k)$.
In \cite[Lemma~B.8]{feng2015hrushovski}, Feng gave a degree bound for $U$ which is double exponential in $n$. In the following lemma, we bound the triangular representation of $U$. The bound we give is polynomial exponential in $n$.

\begin{lem} 
\label{lem: unipotent}
Assume that $n>1$. Let $U$ be a subgroup generated by unipotent elements of $GL_n(k)$. Then $U$ has a triangular representation bounded by \begin{displaymath}
3n^2\big{(} 2n^2(n-1)\big{)}^{148.5n^6}.
\end{displaymath}
\end{lem}

\begin{proof}
By \cite[Lemma~C, page 96]{humphreys}, any one-dimensional subgroup $H$ generated by unipotent elements of $GL_n(k)$ has the form
\begin{displaymath}
H=\{ I_n+Mx+\frac{M^2}{2!}x^2+ \dots +\frac{M^{n-1}}{(n-1)!}x^{n-1} : x \in \mathbb{C} \}
\end{displaymath}
where $M \in \text{Mat}_n(k)$ with $M^n=0$.
By \cite[Proposition, page 55]{humphreys}, $U$ is a product of at most $2\dim{(U)}$ one-dimensional subgroups generated by unipotent elements.
Hence,
\begin{displaymath}
U= \prod\limits_{i=1}^{2\dim{(U)}}{H_i}
\end{displaymath}
where $H_i = \{  I_n+M_ix_i+\frac{M_i^2}{2!}x_i^2+ \dots +\frac{M_i^{n-1}}{(n-1)!}x_i^{n-1}: x_i \in \mathbb{C} \}$ is a one-dimensional subgroup generated by unipotent elements of $GL_n(k)$ and $M_i \in \text{Mat}_n(k)$ with $M_i^n=0$.
Since $\dim{(U)} \le n^2$, the defining equations of $U$
contain at most $3n^2$ variables and have degrees not greater than $2n^2(n-1)$. 
By Lemma \ref{lem:elimination}, the ideal generated by the defining equations of $U$ has a triangular representation bounded by 
\[
3n^2\big{(} 2n^2(n-1)\big{)}^{5.5(3n^2)^3} = 3n^2\big{(}2n^2(n-1)\big{)}^{148.5n^6}.
\qedhere
\]
\end{proof}

Jordan in \cite{Jordan1877} proved that there exists a positive integer $J(n)$ depending on $n$ such that every finite subgroup of $GL_n(k)$ contains a normal abelian subgroup of index at most $J(n)$.
Schur in \cite{schur1911gruppen} provided an explicit bound which is 
\begin{displaymath} 
J(n) \leq \big{(} \sqrt{8n}+1\big{)}^{2n^2}-\big{(} \sqrt{8n}-1\big{)}^{2n^2}.
\end{displaymath}
We use Schur's bound in our computations.
Assume that $n>1$.
Let
\begin{displaymath}
D=3n^2\big{(}2n^2(n-1)\big{)}^{148.5n^6} \tag{3} \label{inequalityD:2},
\end{displaymath}

\begin{displaymath}
d_{1} = \max\limits_{i} \bigg\{{{n^2+D \choose D} \choose i}^2 \bigg\} \tag{4}
\label{inequalityd:1}, 
\end{displaymath}

\begin{displaymath}
d_{2} = d_1D{n^2+D \choose D} \tag{5} \label{inequalityd:3},
\end{displaymath}

\begin{displaymath}
d_{3} = n\Bigg{(} d_{2} (d_{1}^{2}+1)\max\limits_{i}\bigg\{{d_{1}^{2}+1 \choose i}^{2}\bigg\}\Bigg{)}^{5.5n^3} \tag{6}
\label{inequalityd:4}, 
\end{displaymath}
and 
\begin{displaymath}
\bar{d} = J\Bigg(\max\limits_{i}\bigg\{{d_{1}^{2}+1 \choose i}^{2}\bigg\}\Bigg) \tag{7}
\label{inequality: 5}.
\end{displaymath}
Next we give numerical bounds for $D, d_{1}, d_{2}, d_{3}$ and $\bar{d}$ which will be used in the following theorems.
Since
\begin{displaymath}
D=3n^2\big{(}2n^2(n-1)\big{)}^{148.5n^6} \le 3n^2(2n^3)^{148.5n^6}
\end{displaymath}
and
\begin{displaymath}
{{n^2+D}  \choose {n^2}} \le \bigg{(} \frac{e(n^2+D)}{n^2}\bigg{)}^{n^2} \le (e+3e(2n^3)^{148.5n^6})^{n^2}
\le 18^{n^2}(2n^3)^{148.5n^8},
\end{displaymath}

\begin{displaymath}
d_{1} \le \big(2^{n^2+D \choose D}\big)^2 \le \big(2^{18^{n^2}(2n^3)^{148.5n^8}}\big)^2 
\le \big(2^{(2n^3)^{149n^8}}\big)^2 \le 4^{(2n^3)^{149n^8}},
\end{displaymath}

\[
d_{2} \le 4^{(2n^3)^{149n^8}} 3n^2(2n^3)^{148.5n^6} 18^{n^2}(2n^3)^{148.5n^8} \\
= 3n^2 18^{n^2} 4^{(2n^3)^{149n^8}}  (2n^3)^{(148.5n^8+148.5n^6)},
\]

\[
\begin{split}
d_{3} 
\le &n\bigg(3n^218^{n^2}(2n^3)^{148.5n^8+148.5n^6} 4^{(2n^3)^{149n^8}} \big(16^{(2n^3)^{149n^8}} +1 \big) 4^{\big(16^{(2n^3)^{149n^8}} +1 \big)} \bigg)^{5.5n^3}\\
\le &n\bigg((2n^3)^{149n^8+149n^6} 4^{(2n^3)^{149n^8}} \big(16^{(2n^3)^{149n^8}} +1 \big) 4^{\big(16^{(2n^3)^{149n^8}} +1 \big)} \bigg)^{5.5n^3}\\
\le &n\bigg(2(2n^3)^{149n^8+149n^6} 4^{(2n^3)^{149n^8}} 16^{(2n^3)^{149n^8}} 4^{\big(16^{(2n^3)^{149n^8}} +1 \big)} \bigg)^{5.5n^3} \\ 
\le &n\bigg(8(2n^3)^{149n^8+149n^6} 4^{(2n^3)^{149n^8}} 16^{(2n^3)^{149n^8}} 4^{16^{(2n^3)^{149n^8}}} \bigg)^{5.5n^3} \\
\le &n\big(8^{16^{(2n^3)^{149n^8}}} \big)^{5.5n^3}
= n8^{5.5n^316^{(2n^3)^{149n^8}}},
\end{split}
\]
and
\[
\begin{split}
\bar{d} 
\le &\big( \sqrt{8 \cdot 4^{d_1^2+1}}+1\big)^{2n^2}-\big(\sqrt{8 \cdot 4^{d_1^2+1}}-1\big)^{2n^2} \\
\le &\big(\sqrt{32} \cdot 2^{8^{(2n^3)^{149n^8}}}+1 \big)^{2n^2} - \big(\sqrt{32} \cdot 2^{8^{(2n^3)^{149n^8}}}-1 \big)^{2n^2}\\
\le &\big(2\sqrt{32} \cdot 2^{8^{(2n^3)^{149n^8}}} \big)^{2n^2} = 2^{10n^2}4^{n^2 8^{(2n^3)^{149n^8}}}
\end{split}
\]
where $n>1$.

\subsection{Complexity of Hrushovski's Algorithm}
\label{sec: 3.2}

In this section, when we say that a family $\mathcal{F}$ of algebraic subgroups of $GL_n(k)$ is bounded by an integer we mean that the triangular representations of all elements in $\mathcal{F}$ are bounded by that integer.
We prove the following theorem and corollaries following the way in which Feng proved \cite[Proposition~B.11, Lemma~B.12, Lemma~ B.13]{feng2015hrushovski}.
But in order to improve the bounds, we replace the use of Gr\"obner bases with the triangular representations.
Our main result is stated as the following theorem.

\begin{thm}
\label{thm:degreekappa3}
Assume that $n>1$. There is an integer 
\begin{displaymath}
\bar{d} \le 2^{10n^2}4^{n^2 8^{(2n^3)^{149n^8}}}
\end{displaymath}
and a family $\mathcal{F}$ of algebraic subgroups of $GL_n(k)$ whose triangular representations are  bounded by 
\begin{displaymath}
d_{3} \le  n8^{5.5n^316^{(2n^3)^{149n^8}}}
\end{displaymath}
with the following properties: for every algebraic subgroup $H'\subseteq GL_{n}(k)$, there exists an algebraic subgroup $H$ of $\mathcal{F}$ such that 
\begin{itemize}
\item []
(a) $(H')^{\circ} \subseteq H$
\item[]
(b) $H \trianglelefteq H'H \subseteq GL_n(k)$
\item[] 
(c) $[H' : H \cap H']= [H'H : H] \leq \bar{d}$
\item[]
(d) Every unipotent element of $H$ is in $(H')^{\circ}$
\end{itemize}
where $(H')^{\circ}$ is the identity
component of $H'$.
\end{thm}

\begin{proof} 
In the first case we assume that $H'$ is a finite subgroup in $GL_n(k)$.
Since every finite subgroup of $GL_n(k)$ contains a normal abelian subgroup of index at most $J(n)$, we choose such a normal abelian subgroup $\tilde{H'} \subseteq H'$. $\tilde{H'}$ is diagonalizable, so $\tilde{H'}$ is in some maximal torus of $GL_n(k)$.
Let $H$ be the intersection of maximal tori containing $\tilde{H'}$ in $GL_n(k)$.
We prove that $H$ satisfies $(a)-(d)$ for $H'$.
$(a)$ is true because of the construction of $H$.
$(b)$ is true because $H$ normalizes $H'$.
$(c)$ is true because 
\begin{displaymath}
[H'H:H]=[H':H\cap{H'}] \le [H':\tilde{H'}]\le J(n).
\end{displaymath}
So we can choose $\bar{d}=J(n)$.
(d) is true because there is only one unipotent element of $H$ which is the identity.
Let $\mathcal{F}$ be the family of all the  intersections of maximal tori in $GL_n(k)$. 
Then $\mathcal{F}$ is the desired family of algebraic subgroups of $GL_n(k)$ with $d_3=1$.
  
In the second case we assume that $H'$ is a subgroup whose identity component is a torus. 
Let $T$ be the intersection of all maximal tori containing $(H')^{\circ}$ in $GL_n(k)$.
Then $T$ has a triangular representation bounded by 1.
Let $S$ be  the normalizer of $T$ in $GL_n(k)$.
By Lemma \ref{lem:boundhom}, there is a homomorphism 
\begin{displaymath}
\tau_{S,T} : S \longrightarrow GL_{n'}(k)
\end{displaymath}
bounded by 
\begin{displaymath}
(n^2+1)\max_{i}{\bigg\{{n^2+1 \choose i}^2\bigg\}}
\end{displaymath}
such that $\ker{(\tau_{S,T})} = T$ and $n' = \max_{i}{\big\{{n^2+1 \choose i}^2\big\}}$. 
Since the identity component of $H'$ is contained in $T$,
$\tau_{S,T}(H')$ is a finite subgroup of $GL_{n'}(k)$. 
Let $\mathcal{F}_1$ be the family of all the intersections of maximal tori of $GL_{n'}(k)$. 
By the first case, there exists $H_{\mathcal{F}_1} \in \mathcal{F}_1$ such that $(a)-(c)$ are true for $\tau_{S,T}(H')$ with $\bar{d}=J(n')$.
Let 
\begin{displaymath}
H = \tau_{S,T}^{-1}\big{(}\tau_{S,T}(S) \cap H_{\mathcal{F}_1}\big{)}.
\end{displaymath}
We prove that $(a)-(d)$ are true for $H$ and $H'$.
Since the identity component $(H')^{\circ}$ of $H'$ is a torus and $T \subseteq H$, $(H')^{\circ} \subseteq H$. 
This proves $(a)$. 
Let $h' \in H'$. 
Then 
\begin{displaymath}
\tau_{S,T}(h'Hh'^{-1})=\tau_{S,T}(h')\big{(}\tau_{S,T}(S) \cap H_{\mathcal{F}_1}\big{)}\tau_{S,T}(h'^{-1}).
\end{displaymath}
Since $H_{\mathcal{F}_1} \trianglelefteq \tau_{S,T}(H')H_{\mathcal{F}_1}$,
\begin{displaymath}
\tau_{S,T}(h')\big{(}\tau_{S,T}(S) \cap H_{\mathcal{F}_1}\big{)}\tau_{S,T}(h'^{-1}) = \tau_{S,T}(S) \cap H_{\mathcal{F}_1} = \tau_{S,T}(H).
\end{displaymath}
So $h'Hh'^{-1} \subseteq H$ and $H \trianglelefteq H'$.
Hence, $H \trianglelefteq H'H \subseteq GL_n(k)$. This proves $(b)$.
Since $H' \subseteq S$ and
$[\tau_{S,T}(H'):H_{\mathcal{F}_1} \cap \tau_{S,T}(H')] \le J(n')$,
\begin{displaymath}
\begin{split}
[H':H \cap H']=[H'H:H]=[\tau_{S,T}(H'H):\tau_{S,T}(H)]
= [\tau_{S,T}(H')\tau_{S,T}(H):\tau_{S,T}(H)]
\\
=[\tau_{S,T}(H'):\tau_{S,T}(H) \cap \tau_{S,T}(H')]=[\tau_{S,T}(H'):H_{\mathcal{F}_1} \cap \tau_{S,T}(H')] \le J(n').
\end{split}
\end{displaymath}
We can choose $\bar{d}=J(n')=J\big{(}\max_{i}{\big\{{n^2+1 \choose i}^2 \big\}}\big{)}$. 
This proves $(c)$.
Let $h \in H$ be a unipotent element. Then $\tau_{S,T}(h)$ is a unipotent element in $H_{\mathcal{F}_1}$. 
Since every element in $H_{\mathcal{F}_1}$ is semi-simple, $\tau_{S,T}(h)=1$. So $h$ must be in $\ker{(\tau_{S,T})} = T$. 
By the definition of $T$, $T$ is in some torus of $GL_n(k)$. So $h=1$. 
Hence, every unipotent element of $H$ is in $(H')^{\circ}$.
This proves $(d)$.
By Lemma \ref{lem: inversebound}, $H$ has a triangular representation bounded by 
\begin{displaymath}
n(n^2+1)^{5.5n^3} \max_{i}\bigg\{{n^2+1 \choose i}^{11n^3}\bigg\}.
\end{displaymath}
Let $\mathcal{F}$ be the family of such subgroups $H$. Then $\mathcal{F}$ is the desired family with
\begin{displaymath}
d_3 \le n(n^2+1)^{5.5n^3} \max_{i}\bigg\{{n^2+1 \choose i}^{11n^3}\bigg\}.
\end{displaymath}
The general case is proved as follows.
Let $H'_u$ be the intersection of kernels of all characters of $(H')^{\circ}$. $(H')^{\circ}$ is a connected subgroup of $GL_n(k)$, so $H'_u$ is generated by all unipotent elements by \cite[Lemma~B.10]{feng2015hrushovski}. 
By Lemma \ref{lem: unipotent}, $H'_u$ has a triangular representation bounded by $D$. 
Let $N$ be the normalizer of $H'_u$ in $GL_n(k)$. 
By Lemma $\ref{lem:boundhom}$, there exists a homomorphism 
\begin{displaymath}
\tau_{N,H'_u}:N \longrightarrow GL_{d_1}(k)
\end{displaymath}
bounded by 
\begin{displaymath}
d_2 = d_1D{n^2+D \choose D}
\end{displaymath}
such that $\ker(\tau_{N,H'_u})= H'_u$ and
\begin{displaymath}
d_1 = \max\limits_{i} \bigg\{{{n^2+D \choose D} \choose i}^{2}\bigg\}.
\end{displaymath} 
The identity component of ${\tau_{N,H'_u}(H')}$ is a torus in $GL_{d_{1}}(k)$, by the second case, there exists $H'' \subseteq GL_{d_{1}}(k)$ whose triangular representation bounded by
\begin{displaymath}
(d_{1}^{2}+1)\max\limits_{i}\bigg\{{d_{1}^{2}+1 \choose i}^{2}\bigg\}
\end{displaymath}
such that $(a)-(d)$ are true for $\tau_{N,H'_u}(H')$ with 
\begin{displaymath}
\bar{d}=J\Bigg{(}\max\limits_{i}\bigg\{{d_{1}^{2}+1 \choose i}^{2}\bigg\}\Bigg{)} \le 2^{10n^2}4^{n^2 8^{(2n^3)^{149n^8}}}.
\end{displaymath}
Let 
\begin{displaymath}
H={\tau_{N,H'_u}}^{-1}\big{(}H'' \cap \tau_{N,H'_u}(N)\big{)}. 
\end{displaymath}
By Lemma \ref{lem: inversebound}, $H$ has a triangular representation bounded by 
\begin{displaymath}
d_{3}  = n\Bigg{(} d_{2}
(d_{1}^{2}+1)\max\limits_{i}\bigg\{{d_{1}^{2}+1 \choose i}^{2}\bigg\}\Bigg{)}^{5.5n^3}  \le  n8^{5.5n^316^{(2n^3)^{149n^8}}}.
\end{displaymath}
By a similar argument in the proof of \cite[Proposition~B.11]{feng2015hrushovski}, $(a)-(d)$ are true for $H$ and $H'$.
Let $\mathcal{F}$ be the family of such algebraic subgroups $H$. Then $\mathcal{F}$ is the desired family with 
\[
d_{3} \le  n8^{5.5n^316^{(2n^3)^{149n^8}}}.
\qedhere
\]
\end{proof}

\begin{cor} \label{cor: degreebound}
Assume that $n>1$. There exists a family $\bar{\mathcal{F}}$ of algebraic subgroups of $GL_n(k)$ whose triangular representations are bounded by 
\begin{displaymath}
d_{3}
\le n8^{5.5n^316^{(2n^3)^{149n^8}}}
\end{displaymath}
such that for any algebraic subgroup $H' \subseteq GL_n(k)$ there exists $\bar{H}$ of $\bar{\mathcal{F}}$ such that $H' \subseteq \bar{H}$ and every unipotent element of $\bar{H}$ is in $(H')^{\circ}$.
\end{cor}

\begin{proof}
Let $\bar{\mathcal{F}}=\{\bar{H}: \ there \ exists \ H \in \mathcal{F} \ such \ that \ H\trianglelefteq \bar{H} \ and \ [\bar{H}:H] \le \bar{d} \}$.
Let $H' \subseteq GL_n(k)$ be an algebraic subgroup. By Theorem \ref{thm:degreekappa3}, there is an $H \in \mathcal{F}$ such that $(a)-(d)$ are true for $H$ and $H'$. 
Let $\bar{H} = H'H$. 
By $(b)$ and $(c)$ in Theorem \ref{thm:degreekappa3}, $\bar{H} \in \bar{\mathcal{F}}$.
By $(d)$ in Theorem \ref{thm:degreekappa3}, every unipotent element of $H$ is in $(H')^{\circ}$.
Since every unipotent element of $\bar{H}$ is in $\bar{H}^{\circ} \subseteq H^{\circ}$.
Hence, every unipotent element of $\bar{H}$ is in $(H')^{\circ}$.
Since $\bar{H}$ is the union of the cosets of some element in $\mathcal{F}$ and every element of $\mathcal{F}$ has a triangular representation bounded by $d_3$, by Lemma \ref{lem: IJbound}, we have that $\bar{H}$ has a triangular representation bounded by 
\begin{displaymath}
d_{3}
\le n8^{5.5n^316^{(2n^3)^{149n^8}}}.
\end{displaymath}
Therefore, $\bar{\mathcal{F}}$ is bounded by $d_{3}
\le n8^{5.5n^316^{(2n^3)^{149n^8}}}$.
\end{proof}

\begin{cor} \label{cor: Proto}
Let $\bar{\mathcal{F}}$ be the family in Corollary \ref{cor: degreebound}. Then for any algebraic subgroup $H' \subseteq GL_n(k)$, there exists $\bar{H}$ of $\bar{\mathcal{F}}$ such that 
\begin{displaymath}
(\bar{H}^{\circ})^{t} \trianglelefteq (H')^{\circ} \subseteq H' \subseteq \bar{H}.
\end{displaymath}
\end{cor}

\begin{proof}
By Corollary \ref{cor: degreebound}, there exists $\bar{H}$ of $\bar{\mathcal{F}}$ such that $H' \subseteq \bar{H}$ and every unipotent element of $\bar{H}$ is in $(H')^{\circ}$. 
Since $\bar{H}^{\circ}$ is a connected subgroup of $GL_n(k)$, $(\bar{H}^{\circ})^{t}$ is generated by all unipotent elements in $\bar{H}^{\circ}$ by \cite[Lemma~B.10]{feng2015hrushovski}. 
Since $(\bar{H}^{\circ})^t \trianglelefteq \bar{H}^{\circ}$ and every unipotent element in $\bar{H}^{\circ}$ is in $(H')^{\circ}$, $(H^{\circ})^{t} \trianglelefteq (H')^{\circ}$.
Therefore, 
$(\bar{H}^{\circ})^{t} \trianglelefteq (H')^{\circ} \subseteq H' \subseteq \bar{H}$.
\end{proof}

\begin{rem}
Since the differential Galois group of $(\ref{equation:1})$ is an algebraic subgroup in $GL_n(k)$, by Corollary \ref{cor: Proto}, there exists an algebraic subgroup $\bar{H}$ bounded by $d_3$ such that $(\bar{H}^{\circ})^{t} \trianglelefteq (H')^{\circ} \subseteq H' \subseteq \bar{H}$. By the definition of the proto-Galois group, this algebraic subgroup $\bar{H}$ is a proto-Galois group of $(\ref{equation:1})$. 
\end{rem}

\section{Comparison}
\label{sec: 4}
We compute $\bar{d}$ and $d_3$ explicitly for $n=2$. We plug in $n=2$ to the equations (\ref{inequalityD:2}), (\ref{inequalityd:1}), (\ref{inequalityd:3}), (\ref{inequalityd:4}), and (\ref{inequality: 5}) instead of the formulas in Theorem \ref{thm:degreekappa3} and Corollary \ref{cor: degreebound} to do calculations, which would give us more refined bounds. Feng in \cite{feng2015hrushovski} roughly estimated that $\bar{d}$ is quintuply exponential in $n$ and $d_3$ is sextuply exponential in $n$, but he did not give numerical bounds for them. In order to compare our bounds with the ones in \cite{feng2015hrushovski}, we also give numerical bounds in \cite[Proposition~11, Proposition~14]{feng2015hrushovski}.
In \cite[Proposition~11, Proposition~14]{feng2015hrushovski},
$d_3$ is denoted as $\tilde{d}$ and $\bar{d}$ is denoted as $I(n)$ respectively. The numerical bounds of $\tilde{d}$ and $I(n)$ are as follows:
\begin{displaymath}
\tilde{d} \le 32^{2^{2^{2^{(2n)^{2^{(24n^2)}}}}}},
\\
I(n) \le 4^{2^{2^{(2n)^{2^{(12n^2)}}}}}.
\end{displaymath}
When $n=2$, 
\begin{displaymath}
\bar{d} \le 2^{2^{2^{2^{18}}}}, 
I(n) \le 2^{2^{2^{2^{2^{96}}}}},
\end{displaymath}
and
\begin{displaymath}
d_3 \le 2^{2^{2^{2^{18}}}},  
\tilde{d} \le 2^{2^{2^{2^{2^{2^{194}}}}}}.
\end{displaymath}

\begin{ack}
This work has been partially supported by the NSF grants CCF-0952591, CCF-1563942, and DMS-1413859 and by PSC-CUNY grant \#60098-00 48.
I am grateful to Eli Amzallag and Alexey Ovchinnikov for discussions on this work.
\end{ack}

\bibliographystyle{apalike}
\bibliography{main}

\end{document}